\documentclass[a4paper,11pt]{amsart}
 \usepackage[arrow,matrix]{xy}
\usepackage{amsmath,amssymb,amscd,bbm,amsthm,mathrsfs,dsfont,chemarrow}
\usepackage{graphicx}
 \numberwithin{equation}{section}
\newtheorem{theorem}{Theorem}[section]
\newtheorem{lemma}{Lemma}[section]
\newtheorem{corollary}{Corollary}[section]
\newtheorem{proposition}{Proposition}[section]

\theoremstyle{definition}

\theoremstyle{remark} \theoremstyle{example}

\usepackage{geometry}
\begin{document}\numberwithin{equation}{section}
\title[On Einstein Matsumoto metrics]{On Einstein Matsumoto metrics}
\author{Yi-Bing Shen and Xiaoling Zhang} \thanks{Supported partially by NNSFC(No. 11171297)}.

 \begin{abstract}
In this paper, the necessary and sufficient conditions for Matsumoto metrics $F=\frac{\alpha^2}{\alpha-\beta}$ to be Einstein are given. It is shown that if the length of $\beta$ with respect to $\alpha$ is constant, then the  Matsumoto metric $F$ is an Einstein metric if and only if $\alpha$ is Ricci-flat and $\beta$ is parallel with respect to $\alpha$. A nontrivial example of Ricci flat Matsumoto metrics is given.
\end{abstract}
\maketitle

\section{Introduction}

Let $F=F(x,y)$ be a Finsler metric on an n-dimensional manifold $M$. $F$ is called an Einstein metric with Einstein scalar $\sigma$ if its Ricci curvature  $Ric$ satisfies
\begin{equation}\label{xa1}
Ric =\sigma  F^2,
\end{equation}
where  $\sigma=\sigma(x)$ is a scalar function on $M$. In particular,  $F$ is said to be Ricci constant (resp. Ricci flat) if $\sigma =$const. (resp. $\sigma=0$) in \eqref{xa1}. (\cite{AZ},\cite{bao1}).

An important class of Finsler metrics is so called $(\alpha,\beta)$-metrics, which are iteratively appearing in physical studies, and are expressed in the form of $F=\alpha\phi(s)$, $s=\frac{\beta}{\alpha}$, where  $\alpha=\sqrt{a_{ij}(x)y^iy^j}$ is a  Riemannian metric and $\beta=b_i(x)y^ i$ is a 1-form. The $(\alpha,\beta)$-metrics with $\phi=1+s$ are called Randers metrics. D. Bao and C. Robles have characterized Einstein Randers metrics, and shown that every Einstein Randers metric is necessarily Ricci constant in dimension $n \geq 3$. When $n=3$, a Randers metric is Einstein if and only if it is of constant flag curvature, see \cite{bao1}.\\

For non-Randers  $(\alpha,\beta)$-metrics $F$ with a polynomial function $\phi(s)$ of degree greater than 2, it was proved that $F$ is an Einstein metric if and only if it is Ricci-flat (\cite{cheng}). An  $(\alpha,\beta)$-metric with $\phi=s^{-1}$ is called a Kropina metric. It was shown that a Kropina metric $F=\frac{\alpha^2}{\beta}$ is an Einstein metric if and only if $h$ is an Einstein metric and $W$ a unit Killing form with respect to $h$, where $(h,W)$ is the navigation data of $F$ (\cite{zhang}).\\

The Matsumoto metric is an interesting $(\alpha,\beta)$-metric with $\phi=1/(1-s)$,  introduced by using gradient of slope, speed and gravity in \cite{mat}. This metric formulates the model of a Finsler space. Many authors  (\cite{aik,mat,park}, etc) have studied this metric by different perspectives.\\

The present paper is devoted to study Einstein Matsumoto metrics, and main results are as follows.

\begin{theorem}\label{xxa1}
Let $F=\frac{\alpha^2}{\alpha-\beta}$ be a non-Riemann Matsumoto metric on an n-dimensional manifold $M$, $n\geq2$. Then $F$ is an Einstein metric if and only if the followings hold\\
1) $\alpha$ is an Einstein metric, i.e., $\overline{Ric}=\lambda\alpha^2$,\\
2) $\beta$ is a conformal form with respect to $\alpha$, i.e., $r_{00}=c\alpha^2$,\\
3) \begin{equation}\label{xaa1}0=\lambda\alpha^2+2T^k_{\,\,\,|k}-y^jT^k_{\,\,\,.\,k|j}+2T^jT^k_{\,\,\,.\,j\,.\,k}-T^k_{\,\,\,.\,j}T^j_{\,\,\,.\,k}
-\sigma(x)\frac{\alpha^4}{(\alpha-\beta)^2}, \end{equation}
where $\overline{Ric}$ denotes the Ricci curvature of $\alpha$, $\lambda=\lambda(x),c=c(x)$ are functions on $M$, $$T^i=
-\frac{\alpha^2}{2\beta-\alpha}s^i_{\,\,0}
-\frac{3\alpha^3}{3\beta-(2b^2+1)\alpha}(\frac{2
}{2\beta-\alpha}s_0+c)b^i
+\frac{\alpha(4\beta-\alpha)}{2(3\beta-(2b^2+1)\alpha)}(\frac{2
}{2\beta-\alpha}s_0+c)y^i,$$  and $"|"$ and  $"."$  denote the horizontal covariant derivative and vertical covariant derivative with respect to $\alpha$, respectively.
\end{theorem}
Notations here can be referred to \eqref{xba1} and \eqref{xba2} below.

{\bf {Remark.}}\, M. Rafie-Rad, etc., also discussed Einstein Matsumoto metrics. Unfortunately, the computation and results in \cite{raf1} are wrong because they neglected $b^2$ in $a_i\, (i=0,\ldots,14)$. Theorem \ref{xxa1} is the corrected version of Theorem 1 in \cite{raf1}.\\

\begin{theorem}\label{xxa2}
Let $F=\frac{\alpha^2}{\alpha-\beta}$ be a non-Riemannian Matsumoto metric on an n-dimensional manifold $M$, $n\geq3$. Suppose the length of $\beta$ with respect to $\alpha$ is constant. Then $F$ is an Einstein metric if and only if $\alpha$ is Ricci-flat and $\beta$ is parallel with respect to $\alpha$. In this case, $F$ is Ricci-flat.
\end{theorem}

\begin{theorem}\label{xxa3}
Let $F=\frac{\alpha^2}{\alpha-\beta}$ be a non-Riemannian Matsumoto metric on an $n$-dimensional manifold $M$, $n\geq3$. Suppose $\beta^{\sharp}$, which is dual to $\beta$, is a homothetic vector field, i.e., $r_{00}=c\alpha^2$, where $c=constant$. Then $F$ is an Einstein metric if and only if $\alpha$ is Ricci-flat and $\beta$ is parallel with respect to $\alpha$. In this case, $F$ is Ricci-flat.
\end{theorem}

For an $(\alpha,\beta)$-metric, the form $\beta$ is said to be Killing (resp. closed) form if $r_{ij}=0$\,\,(resp. $s_{ij}=0$). $\beta$ is said to be a constant Killing form if it is a Killing form and has constant length with respect to $\alpha$, equivalently $r_{ij}=0$ and $s_i=0$.\\

{\bf {Remark.}}\, B. Rezaei, etc., discussed Einstein Matsumoto metrics with constant Killing form in \cite{rez}. Meanwhile, they got wrong results. Theorem \ref{xxa2} and Theorem \ref{xxa3} generalize their study.\\

For the S-curvature with respect to the Busemann-Hausdorff volume form (\cite{CHEN}), we have following
\begin{theorem}\label{xxa4}
Let $F=\frac{\alpha^2}{\alpha-\beta}$ be a non-Riemannian Matsumoto metric on an n-dimensional manifold $M$, $n\geq2$. Then $S$-curvature vanishes if and only if $\beta$ is a constant Killing form.
\end{theorem}\par
From above theorems, we can easily get the following
\begin{corollary}\label{xxa5}
Let $F=\frac{\alpha^2}{\alpha-\beta}$ be a non-Riemannian Matsumoto metric on an n-dimensional manifold $M$, $n\geq3$. Suppose $F$ is an Einstein metric. Then $S$-curvature vanishes if and only if $\alpha$ is Ricci-flat and $\beta$ is parallel with respect to $\alpha$. In this case, $F$ is Ricci-flat.
\end{corollary}

The content of this paper is arranged as follows. In \S \ref{xxx2} we introduce essential curvatures of Finsler metrics, as well as notations and conventions. And we give the spray coefficients of Matsumoto metrics.  Theorem \ref{xxa1} is proved in \S \ref{xxx3}. In \S \ref{xxx4}, we first give the necessary and sufficient conditions for Matsumoto metrics to be Einstein under the hypothesis condition that $\beta$ is a constant Killing form. By using it, Theorem \ref{xxa2} and Theorem \ref{xxa3} are proved. A nontrivial example of Ricci flat Matsumoto metrics is shown. By the way, we characterize  Matsumoto metrics $F$ with constant Killing form $\beta$, which are of constant flag curvature. In \S \ref{xxxs} we investigate the $S$-curvature of Matsumoto metrics and Theorem \ref{xxa4} is proved. In the last Section \S \ref{xxx6} we list the coefficients appeared in the proof of Theorem \ref{xxa1}.

\vspace{1cm}

 \section{Preliminaries}\label{xxx2}
Let $F$ be a Finsler metric on an $n$-dimensional manifold $M$ and $G^i$ the geodesic coefficients of $F$, which are defined by
\begin{equation*}
G^i:=\frac{1}{4}g^{il}\{[F^2]_{x^ky^l}y^k-[F^2]_{ x^l}\}.
\end{equation*}
For any $x\in M$ and $y\in T_xM \setminus\{0\}$, the Riemann curvature $\textbf{R}_y:= R^i_{\,\,\,k}\frac{\partial}{\partial x^i}\bigotimes dx^k$ is defined by
\begin{equation}\label{xba4}
 R^i_{\,\,\,k}:=2\frac{\partial G^i}{\partial x^k}-\frac{\partial^2 G^i}{\partial x^j \partial y^k}y^j
  +2G^j \frac{\partial^2G^i}{\partial
 y^j \partial y^k}-\frac{\partial G^i}{\partial y^j}\frac{\partial G^j}{\partial
 y^k}.
\end{equation}
Ricci curvature is the trace of the Riemann curvature, which is defined by
\begin{equation}\label{xba5}
Ric:= R^k_{\,\,\, k}.
\end{equation}

A Finsler metric $F$ is called an Einstein metric with Einstein scalar $\sigma$ if
\begin{equation}\label{xb1}
Ric=\sigma F^2,
\end{equation}where  $\sigma=\sigma(x)$ is a scalar function on $M$. In particular,  $F$ is said to be Ricci constant (resp. Ricci flat) if $F$ satisfies \eqref{xb1} where $\sigma =$const. (resp. $\sigma=0$).

By definition, an $(\alpha,\beta)$-metric on $M$ is expressed in the form $F=\alpha\phi(s),  s=\frac{\beta}{\alpha}$,  where $\alpha=\sqrt{a_{ij}(x)y^iy^j}$  is a positive definite Riemannian metric and $\beta=b_i(x)y^ i$ is a 1-form. It is known that $(\alpha,\beta)$-metric with $||\beta_x||_{\alpha}<b_0$ is a Finsler metric if and only if $\phi=\phi(s)$ is a positive smooth function on an open interval $(-b_0,b_0)$ satisfying the following condition (see \cite{CHEN})
\begin{equation*}
\phi(s)-s\phi'(s)+(b^2-s^2)\phi''(s)>0, ~~~~\forall|s|\leq b<b_0.
\end{equation*}\label{zhengding0}

Let
\begin{equation}\begin{aligned}\label{xba1}
r_{ij}=\frac{1}{2}(b_{i|j}+b_{j|i}),~~~~s_{ij}=\frac{1}{2}(b_{i|j}-b_{j|i}),~~~~~~~~~~~~~~~~~~~~~~~~~~~~~~~~
\end{aligned}\end{equation}
where $ "|" $ denotes the horizontal covariant derivative with respect to  $\alpha$.$^{[5]}$  Denote
\begin{equation}\begin{aligned}\label{xba2}
&r^i_{\,\,j}:= a^{ik}r_{kj}, ~~~~~~~~~~~ r_j:=b^ir_{ij}, ~~~~~~~~~~~ r:=r_{ij}b^ib^j=b^jr_j, ~~~~~~~~~~~r^i:=a^{ij}r_j\\
&s^i_{\,\,j}:= a^{ik}s_{kj}, ~~~~~~~~~~~ s_j:=b^is_{ij}, ~~~~~~~~~~~ s^i:=a^{ij}s_j,\\
&r_{i0}:=r_{ij}y^j, ~~~~~~~~~~~ r_{00}:=r_{ij}y^iy^j, ~~~~~~~~~~~  r_0:=r_iy^i, ~~~~~~~~~~~
s_{i0}:=s_{ij}y^j, ~~~~~~~~~~~s^i_{\,\,0}:=s^i_{\,\,j}y^j, ~~~~~~~~~~~  s_0:=s_iy^i,\\
\end{aligned}\end{equation}
where $(a^{ij}):=(a_{ij})^{-1}$ and $b^i:=a^{ij}b_j$.

 Let $G^i$ and $\bar{G^i}$ be the geodesic coefficients of $F$ and $\alpha$, respectively. Then we have the following \par
\begin{lemma}[\cite{li}] \label{xxb1}
For an $(\alpha,\beta)$-metric $F =\alpha\phi(s),$ $s=\frac{\beta}{\alpha}$, the geodesic coefficients $G^i$ are given by
\begin{equation}\label{xb2}
G^i=\bar{G^i} +\alpha Q s^i_{\,\,0} +\Psi(r_{00}-2\alpha
Qs_0)b^i+\frac{1}{\alpha}\Theta(r_{00}-2\alpha Qs_0)y^i,
\end{equation}
where
\begin{equation*}\begin{aligned}
&Q:=\frac{\phi'}{\phi-s\phi'},\\
&\Psi:=\frac{\phi''}{2[\phi-s\phi'+(b^2-s^2)\phi'']},\\
&\Theta:=\frac{\phi\phi'-s(\phi\phi''+\phi'\phi')}{2\phi[\phi-s\phi'+(b^2-s^2)\phi'']}.\\
\end{aligned}\end{equation*}
\end{lemma}\par

From now on, we consider a special kind of $(\alpha,\beta)$-metrics which is called Matsumoto-metrics with the form
\begin{equation*}
F=\alpha
\phi(s),~~~\phi(s):=\frac{1}{1-s},~~~s=\frac{\beta}{\alpha}.
\end{equation*}

Let $b_0$  be the largest number such that for any $s$ with $|s|\leq b<b_0$. From Lemma 3.1 in \cite{li}, we have known that $F$ is a Finsler metric if and only if $b=||\beta_x||_{\alpha}<b_0=\frac{1}{2}$. So we always assume that
$\phi$ satisfies  this condition.

Now we get the spray coefficients of Matsumoto metrics by using Lemma \ref{xxb1}.
\par

\begin{proposition}\label{xxc1}
For the Matsumoto metric $F=\frac{\alpha^2}{\alpha-\beta}$, its geodesic coefficients are
\begin{equation}\label{xc1}\begin{aligned} G^i=&\bar{G}^i
-\frac{\alpha}{2s-1}s^i_{\,\,0} -\frac{1}{3s-2b^2-1}(\frac{2\alpha
}{2s-1}s_0+r_{00})b^i\\
 &+\frac{4s-1}{2(3s-2b^2-1)}(\frac{2\alpha
}{2s-1}s_0+r_{00})\frac{y^i}{\alpha}.\\
\end{aligned}\end{equation}\end{proposition}
\begin{proof}
For $\phi(s)=\frac{1}{1-s}$ and by a direct computation, we can obtain \eqref{xc1} from \eqref{xb2}.
\end{proof}

\vspace{1cm}

\section{Einstein Matsumoto metrics }\label{xxx3}
\vspace{4mm}
By using Proposition \ref{xxc1}, we now prove Theorem \ref{xxa1}.
\vspace{4mm}

{\it Proof of Theorem \ref{xxa1}}\par
 Let\begin{equation*}
G^i=\bar{G}^i +T^i,
\end{equation*}
where\begin{equation*}\begin{aligned} T^i=&
-\frac{\alpha}{2s-1}s^i_{\,\,0} -\frac{1}{3s-2b^2-1}(\frac{2\alpha
}{2s-1}s_0+r_{00})b^i+\frac{4s-1}{2(3s-2b^2-1)}(\frac{2\alpha
}{2s-1}s_0+r_{00})\frac{y^i}{\alpha}.\\
\end{aligned}\end{equation*}
Thus by \eqref{xba4}, \eqref{xba5} and \eqref{xb2}, the Ricci curvature of $F$ is related to the Ricci curvature of $\alpha$ by
\begin{equation}\label{xd0}
Ric=\overline{Ric}+2T^k_{\,\,\,|k}-y^jT^k_{\,\,\,.\,k|j}+2T^jT^k_{\,\,\,.\,j\,.\,k}-T^k_{\,\,\,.\,j}T^j_{\,\,\,.\,k},
\end{equation}
where $\overline{Ric}$ denotes the Ricci curvature of $\alpha$, $"|"$ and  $"."$  denote the horizontal covariant derivative and vertical covariant derivative with respect to $\alpha$, respectively.$^{[5]}$\par
So the necessary and sufficient condition for the Matsumoto metric to be an Einstein metric is
\begin{equation}\label{xd1}\begin{aligned}
0&=Ric-\sigma(x)F^2\\
&=\overline{Ric}+2T^k_{\,\,\,|k}-y^jT^k_{\,\,\,.\,k|j}+2T^jT^k_{\,\,\,.\,j\,.\,k}-T^k_{\,\,\,.\,j}T^j_{\,\,\,.\,k}
-\sigma(x)\frac{\alpha^2}{(1-s)^2}.
\end{aligned}\end{equation}

Multiplying both sides of \eqref{xd1} by $\alpha^{12}(s-1)^2(2s-1)^4(3s-2b^2-1)^4$ and by a quite long computational procedure using Maple program, we obtain
\begin{equation}\label{xd2}
0=\sum^{14}_{m=0}t_m \,\alpha^m,
\end{equation} where $t_m,m=0,1,...,14$ are as follows
\begin{equation*}\begin{cases}
\begin{aligned}
t_0&= 144(8n-11)\beta^{10}r_{00}^2,\\
& \ldots\\
t_{14}&=
-(1+2b^2)^4s^j_{\,\,\,k}s^k_{\,\,\,j}-(1+2b^2)^4\sigma-4(1+2b^2)^3s^ks_k.\\
\end{aligned}\end{cases}
\end{equation*}
All the coefficients of $\alpha$ are tedious, listed in Appendix \S \ref{xxx6}.

If we replace $y$ by $-y$, then $t_{2m}(-y)=t_{2m}(y)$ and $t_{2\bar{m}+1}(-y)=-t_{2\bar{m}+1}(y)$ for $m=0,\ldots,7$ and $\bar{m}=0,\ldots,6$. Hence \eqref{xd2} is equivalent to the following
\begin{equation}\begin{cases}\label{xd3}
\begin{aligned}
&0=t_0+t_2\alpha^2+t_4\alpha^4+t_6\alpha^6+t_8\alpha^8
+t_{10}\alpha^{10}+t_{12}\alpha^{12}+t_{14}\alpha^{14},\\
&0=t_1+t_3\alpha^2+t_5\alpha^4+t_7\alpha^6+t_9\alpha^8+t_{11}\alpha^{10}+t_{13}\alpha^{12}.
\end{aligned}\end{cases}
\end{equation}

From the first equation of \eqref{xd3}, we know that $\alpha^2$ divides $t_0$. Since $\alpha^2$ is an irreducible polynomial in $y$ and $\beta^{10}$ factors into ten linear terms, it must be the case that $\alpha^2$ divides $r_{00}^2$. Thus $r_{00}=c\alpha^2$ for some function $c=c(x)$, i.e., $\beta$ is a conformal form with respect to  $\alpha$. So it is easy to get
\begin{equation}\label{xd4}\begin{cases}
\begin{aligned}
&r_{00}=c\alpha^2,~~~~~~r_{ij}=ca_{ij},~~~~~r_{0j}=cy_{j},~~~r_{i}=cb_{i},~~~~~~r=cb^2,~~~r^i_{\,\,\,j}=c\delta^i_{\,\,\,j},\\
&r_{0k}s^k_{\,\,\,0}=0,~~~~~r_{0k}s^k=cs_0,~~~~r_{0}=c\beta,~~~s^k_{\,\,\,0}r_{k}=cs_{0},\\
&r_{00|k}=c_k\alpha^2,~~~r_{00|0}=c_0\alpha^2,~~~r^k_{\,\,\,k}=nc,~~~r_{0|0}=c_0\beta+c^2\alpha^2,\\
\end{aligned}\end{cases}
\end{equation}
where $y_i:=a_{ij}y^j, ~~c_k:=\frac{\partial c}{\partial x^k}$ and $c_0:= c_ky^k$.\par

Plugging \eqref{xd4} into the first equation of \eqref{xd3} and removing the common factor $\alpha^2$, we obtain
\begin{equation*}
0=\bar{t}_0+\bar{t}_2\alpha^2+...+\bar{t}_{12}\alpha^{12},
\end{equation*}
where
\begin{equation*}\begin{cases}
\begin{aligned}
\bar{t}_0&=1296\overline{Ric}\beta^{10},\\
\bar{t}_2&=72(225+240b^2+48b^4)\beta^8\overline{Ric}
+72(-151-56b^2+63n+24nb^2)\beta^9c_0\\
& -72(142+56b^2-57n-24nb^2)\beta^8s_{0|0} -144(8n-21)\beta^8s_{0}^2
-288(8n-15)\beta^9s_{0}c\\
& -1296\beta^9s^k_{\,\,\,0|k} -144(-21+8n)\beta^{10}c^2.
\end{aligned}\end{cases}
\end{equation*}\par

Due to the irreducibility of $\alpha$, we have $\alpha^2$ divides $\overline{Ric}$, i.e., there exists some function
$\lambda=\lambda(x)$ such that
\begin{equation}\label{xd5}
 \overline{Ric}=\lambda\alpha^2.
\end{equation}
It implies that $\alpha$ is an Einstein metric.

Plugging \eqref{xd4} and \eqref{xd5} into \eqref{xd1} yields \eqref{xaa1}.

Conversely, plugging \eqref{xd4}, \eqref{xd5} and \eqref{xaa1} into \eqref{xd0} yields \eqref{xd1}, which means that $F$ is an Einstein metric. It completes the proof of Theorem \ref{xxa1}.
\qed\\

{\bf {Remark.}}  For Riemann curvature and Ricci curvature  of $(\alpha,\beta)$-metrics, L. Zhou gave some formulas in \cite{zhou}. However, Cheng has corrected some errors of his formulas in \cite{cheng}. To avoid making such mistakes, we use the definitions \eqref{xba4} and \eqref{xba5} of Riemann curvature and Ricci curvatures to compute them.
\vspace{1cm}
\section{The proofs of Theorem \ref{xxa2} and Theorem \ref{xxa3}}\label{xxx4}

\vspace{6mm}
\begin{lemma}\label{xxe1}
Let $F=\frac{\alpha^2}{\alpha-\beta}$ be a non-Riemann Matsumoto metric on an n-dimensional manifold $M$, $n\geq3$. Suppose $\beta$ is a constant Killing form, i.e., $r_{ij}=0,s_i=0$. Then $F$ is an Einstein metric if and only if $\alpha$ is Ricci-flat and $\beta$ is parallel with respect to $\alpha$. In this case, $F$ is Ricci-flat.
\end{lemma}

\begin{proof}
If $F$ is an Einstein metric, then \eqref{xd2} holds by Theorem \ref{xxa1}. Removing the common factor $\alpha^2(\alpha-2\beta)(3\beta-2b^2\alpha-\alpha)^4$ from \eqref{xd2}, we obtain
\begin{equation*}
\begin{aligned}
0=&-8\overline{Ric}\beta^5 +28\overline{Ric}\beta^4\alpha
+2(-19\overline{Ric}+4s^k_{\,\,\,0|k}\beta)\beta^3\alpha^2
+(-24s^k_{\,\,\,0|k}\beta+25\overline{Ric}+2s_{0k}s^k_{\,\,\,0})\beta^2\alpha^3\\
&+2(-4\overline{Ric}+13s^k_{\,\,\,0|k}\beta-2s_{0k}s^k_{\,\,\,0}+s^j_{\,\,\,k}s^k_{\,\,\,j}\beta^2+4\sigma\beta^2)\beta\alpha^4\\
&+(\overline{Ric}-12s^k_{\,\,\,0|k}\beta+2s_{0k}s^k_{\,\,\,0}-5s^j_{\,\,\,k}s^k_{\,\,\,j}\beta^2-12\sigma\beta^2)\alpha^5\\
&+2(2s^j_{\,\,\,k}s^k_{\,\,\,j}\beta+s^k_{\,\,\,0|k}+3\sigma\beta)\alpha^6
-(s^j_{\,\,\,k}s^k_{\,\,\,j}+\sigma)\alpha^7.\\
\end{aligned}
\end{equation*}
Obviously, the equation above is equivalent to
\begin{equation}\label{xe1}\begin{cases}
\begin{aligned}
0=&-4\overline{Ric}\beta^5
+(-19\overline{Ric}+4s^k_{\,\,\,0|k}\beta)\beta^3\alpha^2+(-4\overline{Ric}+13s^k_{\,\,\,0|k}\beta-2s_{0k}s^k_{\,\,\,0}
+s^j_{\,\,\,k}s^k_{\,\,\,j}\beta^2+4\sigma\beta^2)\beta\alpha^4\\
&+(2s^j_{\,\,\,k}s^k_{\,\,\,j}\beta+s^k_{\,\,\,0|k}+3\sigma\beta)\alpha^6,\\
0=&28\overline{Ric}\beta^4
+(-24s^k_{\,\,\,0|k}\beta+25\overline{Ric}+2s_{0k}s^k_{\,\,\,0})\beta^2\alpha^2+(\overline{Ric}-12s^k_{\,\,\,0|k}\beta+2s_{0k}s^k_{\,\,\,0}-5s^j_{\,\,\,k}s^k_{\,\,\,j}\beta^2-12\sigma\beta^2)\alpha^4
\\
&-(s^j_{\,\,\,k}s^k_{\,\,\,j}+\sigma)\alpha^6.\\
\end{aligned}\end{cases}
\end{equation}\par
From the first equation of \eqref{xe1}, we have $\overline{Ric}=\lambda\alpha^2$ for some function $\lambda=\lambda(x)$ on $M$. Using the Bianchi identity, i.e., $b_{j|k|l}-b_{j|l|k}=b^s\bar{R}_{jskl}$, we obtain
\begin{equation}\label{xe2}
s^l_{\,\,\,k|l}=\lambda b_k.
\end{equation}
Contracting both sides of \eqref{xe2} with $b^k$ and $y^k$, respectively, we have
\begin{equation}\label{xe3}\begin{cases}\begin{aligned}
s^k_{\,\,\,j}s^j_{\,\,\,k}&=-\lambda b^2,\\
s^k_{\,\,\,0|k}&=\lambda\beta.
\end{aligned}\end{cases}\end{equation}

Substituting \eqref{xe3} into \eqref{xe1} yields
\begin{equation}\label{xe4}
\begin{aligned}
0=& (-4\sigma\beta^2+\lambda
b^2\beta^2+6\lambda\beta^2+2s_{0k}s^k_{\,\,\,0})
+(-3\sigma+3\lambda+2\lambda b^2)\alpha^2,\\
\end{aligned}
\end{equation}
and
\begin{equation}\label{xe5}
\begin{aligned}
0=& 2(2\lambda\beta^2+s_{0k}s^k_{\,\,\,0})\beta^2
+(-12\sigma\beta^2+2s_{0k}s^k_{\,\,\,0}+5\lambda
b^2\beta^2+13\lambda\beta^2)\alpha^2+\{-\sigma+(1+b^2)\lambda\}\alpha^4.\\
\end{aligned}
\end{equation}

$3\times\eqref{xe5}-\eqref{xe4}\times\alpha^2$ yields
\begin{equation}\label{xe6}
\begin{aligned}
0=&6(2\lambda\beta^2+s_{0k}s^k_{\,\,\,0})\beta^2
+(-32\sigma\beta^2+4s_{0k}s^k_{\,\,\,0}+14\lambda
b^2\beta^2+33\lambda \beta^2)\alpha^2+\lambda b^2\alpha^4.\\
\end{aligned}
\end{equation}\par

Since $\alpha^2$ is irreducible polynomial of $y$, we assume that
\begin{equation}\label{xe7}
2\lambda\beta^2+s_{0k}s^k_{\,\,\,0}=h\alpha^2
\end{equation}
holds for some function $h=h(x)$ on $M$. Differentiating both sides of \eqref{xe7} with respect to  $y^iy^j$  yields $4\lambda b_ib_j+s_{ik}s^k_{\,\,\,j}+s_{jk}s^k_{\,\,\,i}=2ha_{ij}$. Then contracting it with $b^ib^j$ gives $h=2\lambda b^2$. Thus $s_{0k}s^k_{\,\,\,0}=h\alpha^2-2\lambda\beta^2=2\lambda(b^2\alpha^2-\beta^2)$. Plugging it into \eqref{xe6}, we get
\begin{equation}\label{xe8}
\begin{aligned}
0=(-32\sigma+28\lambda
b^2+25\lambda )\beta^2 + 9\lambda b^2\alpha^2.\\
\end{aligned}
\end{equation}
Hence \eqref{xe8} is equivalent to
\begin{equation}\bigg\{\label{xe9}
\begin{aligned}
0=&-32\sigma+28\lambda
b^2+25\lambda, \\
0=&9\lambda b^2.\\
\end{aligned}
\end{equation}

From the second equation of \eqref{xe9}, we have $\lambda=0$. Plugging it into the first equation of \eqref{xe9} gives $\sigma=0$, i.e., $F$ is Ricci-flat.

Moreover, substituting $\lambda=0$ into \eqref{xe3} yields $s_{ij}=0$. Together with $r_{ij}=0$, we have $b_{i|j}=0$, i.e., $\beta$ is parallel with respect to $\alpha$.

Converse is obvious. It completes the proof of Lemma \ref{xxe1}.
\end{proof}

\vspace{6mm}

It is found that if $\beta$ satisfies $s_i=0$ or $r_{00}=c\alpha^2$, where $c=constant$, then $\beta$ is a constant Killing form when $F$ is Einstein.
Firstly, we prove the following
\begin{theorem}\label{xxe5}
Let $F=\frac{\alpha^2}{\alpha-\beta}$ be a non-Riemann Matsumoto metric on an n-dimensional manifold $M$, $n\geq3$. Suppose $\beta$ satisfies $s_i=0$. Then $F$ is an Einstein metric if and only if $\alpha$ is Ricci-flat and $\beta$ is parallel with respect to $\alpha$. In this case, $F$ is Ricci-flat.
\end{theorem}
\begin{proof}
If $F$ is an Einstein metric, then $r_{00}=c\alpha^2$ and $\overline{Ric}=\lambda \alpha^2$ by Theorem \ref{xxa1}. Plugging $s_i=0, r_{00}=c\alpha^2, \overline{Ric}=\lambda \alpha^2$ into the second equation of \eqref{xd3} yields
\begin{equation}\label{xf1}
\begin{aligned}
0&=432(5-2n)\beta^{10}c_0\\
+&\{96[48n-123+(12n-18)b^2]c^2\beta^9
-3456(2+b^2)\lambda\beta^9
-864b^kc_k\beta^9\\
&-24 [435n-602+(354n-440)b^2+(48n-56)b^4]c_0\beta^8
\}\alpha^2+\ldots.\\
\end{aligned}
\end{equation}\par
From \eqref{xf1}, we have that $\alpha^2$ divides $\beta^9c_0$. Since $\alpha^2$ is irreducible polynomial of $y^i$, we have $c_0=0$, i.e., $c=constant$. Plugging it into the first equation of \eqref{xd3} yields
\begin{equation}\label{xf2}
\begin{aligned}
0=&144(21-8n)\beta^{10}c^2\\
&-8\beta^6(4832\beta^2c^2 -81\beta^2s^j_{\,\,\,k}s^k_{\,\,\,j}
-568nb^2\beta^2c^2 +702\lambda\beta^2 +128b^4\beta^2c^2
\\
&-1177n\beta^2c^2 -324\sigma\beta^2 -64nb^4\beta^2c^2 +432\lambda
b^2\beta^2+1376b^2\beta^2c^2\\
&+270s_{0k}s^k_{\,\,\,0}+216b^2s_{0k}s^k_{\,\,\,0})\alpha^2+\ldots.\\
\end{aligned}
\end{equation}
From \eqref{xf2}, we get $c=0$ for the division reason again. So $\beta$ is a constant Killing form. Thus by Lemma \ref{xxe1}, we get the necessary conditions.

Sufficiency is obvious. It completes the proof of Theorem \ref{xxe5}.

\end{proof}

\vspace{4mm}

{\it Proof of Theorem \ref{xxa2}}\par
\vspace{4mm}
If $F$ is an Einstein metric, then $r_{00}=c\alpha^2$ by Theorem \ref{xxa1}. Thus $r_k=cb_k$. Since the length of $\beta$, with respect to $\alpha$, is constant, we have $0=b^2_{\,\,\,|k}=2(r_k+s_k)$, i.e., $r_k+s_k=0$. Hence we get $cb_k+s_k=0$. Contracting both sides of it with $b^k$ yields that $c=0$. Above all, $r_{00}=0$ and $s_k=0$, i.e., $\beta$ is a constant Killing form. Thus by Lemma \ref{xxe1}, we obtain that $\alpha$ is Ricci-flat and $\beta$ is parallel with respect to $\alpha$.

Conversely, if $\alpha$ is Ricci-flat and $\beta$ is parallel with respect to $\alpha$, then the length of $\beta$, with respect to $\alpha$, is constant. Hence by Lemma \ref{xxe1}, we get $F$ is Einstein. It completes the proof of Theorem \ref{xxa2}.\qed\par

\vspace{4mm}

Note that the condition that $s_k=0$ in Theorem \ref{xxe5}  is weaker than one that the length of $\beta$ is constant (with respect to $\alpha$) in Theorem \ref{xxa2}.
\vspace{3mm}

{\it Proof of Theorem \ref{xxa3}}\par
Assume $F$ is an Einstein metric and $\beta$ is a homothetic form, i.e., $r_{00}=c\alpha^2$, where $c=constant$. Then \eqref{xd3} holds, i.e.,
\begin{equation}\begin{cases}\label{xe33}
\begin{aligned}
&0=t_0+t_2\alpha^2+t_4\alpha^4+t_6\alpha^6+t_8\alpha^8
+t_{10}\alpha^{10}+t_{12}\alpha^{12}+t_{14}\alpha^{14},\\
&0=t_1+t_3\alpha^2+t_5\alpha^4+t_7\alpha^6+t_9\alpha^8+t_{11}\alpha^{10}+t_{13}\alpha^{12},
\end{aligned}\end{cases}
\end{equation}
 where
\begin{equation*}
\begin{aligned}
t_0&= 144(8n-11)c^2\beta^{10}\alpha^4,\\
\end{aligned}
\end{equation*}
and \begin{equation*}
\begin{aligned}
t_2&= 12\{1085n-1439+(792n-1032)b^2+64(n-1)b^4\}c^2\beta^8\alpha^4\\
&+1296\lambda\beta^{10}\alpha^2-288(8n-14)c^2\beta^{10}\alpha^2+864c\beta^9s_0\alpha^2.\\
\end{aligned}
\end{equation*}

For division reason again, we have $\alpha^2$ can divide $\beta(f\beta+gs_0)$, where  $f:=144(8n-11)c^2+1296\lambda-288(8n-14)c^2,g:=864c$. So we have $f\beta+gs_0=0$. Differentiating both sides of it by $y^i$ and contracting it with $b^i$ yields $f=0$. So $g=0$ or $s_0=0$.

$g=0$ implies that $c=0$. Plugging it into $f=0$ yields $\lambda=0$ and $s^k_{\,\,\,0|k}=0$. Substituting all these into \eqref{xe33} yields
\begin{equation}\label{xe34}
\begin{aligned}t_1=0,t_3=-432 (2n-5)\beta^9s_{0|0},
\end{aligned}
\end{equation}
and
\begin{equation}\label{xe35}
\begin{aligned}
t_0=t_2=0,t_4=72\{57n-142+(24n-56)b^2\}\beta^8s_{0|0}-144(8n-21)\beta^8s_0^2.
\end{aligned}
\end{equation}

From \eqref{xe34}, we know that $\alpha^2$ can divide $s_{0|0}$. Plugging it into \eqref{xe35} yields $\alpha^2$ can divide $s_0^2$. That is $s_0^2=k(x)\alpha^2$, which is a contraction unless $t(x)=0$, i.e., $s_0=0$.

Above all, $s_0=0$. This is the just case in Theorem \ref{xxe5}. It completes the proof of Theorem \ref{xxa3}.\qed\\

{\bf {Remark.}}\, B. Rezaei, etc., discussed Einstein Matsumoto metrics with constant Killing form. Meanwhile, they got wrong results. Theorem \ref{xxa2} and Theorem \ref{xxa3} generalize their study and Lemma \ref{xxe1} is the corrected version of Theorem 4.2 in \cite{rez}.

\vspace{4mm}

{\bf {Example.}}\,  Let $(M,\alpha)$ be an $5$-dimensional Riemanian manifold. Consider the Riemannian metric $\alpha=\sqrt{a_{ij}(x)y^iy^j},~(1\leq i,j\leq 5)$, which, in local coordinate $(x^i)$, can be described as follows
\begin{equation*}
(a_{ij})
 =\left(\begin{array}{ccccc}
(x^4)^2&0&0&0&0\\
0&(x^4)^2&0&0&0\\
0&0&(x^4)^{-1}&0&0\\
0&0&0&x^4&0\\
0&0&0&0&1
 \end{array}\right),
 \end{equation*}
where $x^4>0$. A direct computation shows that $\alpha$ is a non-Euclidean Ricci flat metric.
And let $\beta=cy^5$, where $c$ is a nonzero constant and $c^2<\frac 12$. It is easy to check that such a $\beta$ is parallel with respect to $\alpha$, i.e., $b_{i|j}=0$. Define $F=\frac{\alpha^2}{\alpha-\beta}$. Thus by Theorem \ref{xxa2}, we conclude that
$F=\frac{\alpha^2}{\alpha-\beta}$ is a Ricci-flat Matsumoto metric.

\vspace{6mm}
\begin{theorem}\label{xxe2}
Let $F=\frac{\alpha^2}{\alpha-\beta}$ be a non-Riemannian Matsumoto metric on an n-dimensional manifold $M$, $n\geq 3$. Suppose the length of $\beta$ with respect to $\alpha$ is constant. Then $F$ is of constant flag curvature $K$ if and only if the following conditions hold: \\
(1)  $\alpha$ is a flat metric;\\
(2) $\beta$ is parallel with respect to $\alpha$.\\
In this case, $K=0$ and $F$ is locally Minkowskian.
\end{theorem}
\begin{proof}
Suppose that $F$ is of constant flag curvature $K$, i.e.,
\begin{equation*}
R^i_{\,\,\,k}=K(F^2\delta^i_{\,\,\,k}-g_{ij}y^jy^k).
\end{equation*}
Then we have
\begin{equation}\label{ze11}
Ric=\sigma F^2,  \qquad\qquad \sigma:=(n-1)K,
\end{equation}
which means that $F$ is Einstein. Since the length of $\beta$, with respect to $\alpha$, is constant, by Theorem \ref{xxa2}, we get $\alpha$ is Ricci-flat and $\beta$ is parallel with respect to $\alpha$. In this case, $F$ is Ricci-flat, which means that $K=0$. So $G^i=\bar{G}^i$ and $R^i_{\,\,\,k}=\bar{R}^i_{\,\,\,_k}=0$, i.e., $\alpha$ is Euclidean.

Conversely, if $\alpha$ is Euclidean and $\beta$ is parallel with respect to $\alpha$, then $R^i_{\,\,\,k}=0$, i.e., $K=0$. It completes the proof of Theorem \ref{xxe2}.
\end{proof}

{\bf {Remark.}}\, In literature [9], the proof of Theorem 1 depends on Theorem 3, of which the proof includes the assumption condition that the length of $\beta$ with respect to $\alpha$ is constant, see the step 1 in the proof of Theorem 3 (A and $A_i$ (i= 0,1,2, ...) are some constants) in \cite{raf2}. So, Theorem \ref{xxe2} here is the correct version of Theorem 1 in \cite{raf2}. We do not know what happened to the case that the above assumption is canceled?

\vspace{1cm}
\section{$S$ curvature}\label{xxxs}

The $S$-curvature is an important geometric quantity. In this section, we investigate the $S$-curvature of Matsumoto metrics.\par
 For a Finsler metric $F$ and a volume form $dV=\sigma_{F}(x)\,dx$ on an $n$-dimensional manifold $M$, the  $S$-curvature $S$ is given by
\begin{equation}\label{xg1}
S(x, y)=\frac{\partial G^i}{\partial y^i}-y^i\frac{\partial
\ln\sigma_F }{\partial x^i}.
 \end{equation}
The volume form can be the Busemann-Hausdorff volume form $dV_{BH}=\sigma_{BH}dx$ or the Holmes-Thompson volume form
$dV_{HT}=\sigma_{HT}dx$. \par
 To compute the S-curvature, one should first find a formula for the Busemann-Hausdorff volume forms $dV_{BH}$ and the Holmes-Thompson $dV_{HT}$.
\begin{proposition}(pro4.1 in \cite{bacso})\label{xxg1}
Let $F=\alpha\phi(s)$, $s=\frac{\beta}{\alpha}$, be an $(\alpha,\beta)$-metric on an $n-$dimensional manifold $M$.  Denote
\begin{equation}\label{xg2}f(b):=\begin{aligned}\begin{cases}
&\frac{\int^{\pi}_0\sin^{n-2}(t)dt}{\int^{\pi}_0\frac{\sin^{n-2}(t)}{\phi(b\cos
t)^n}dt}~~\qquad\qquad\,\quad\quad\quad~~if~~\quad~~ dV = dV_{BH},\\
&\frac{\int^{\pi}_0\sin^{n-2}(t)T(b\cos t)dt}
{\int^{\pi}_0\sin^{n-2}(t)dt}
~~\qquad\qquad~~~\quad\quad~~if~~\quad~~ dV = dV_{HT}.\\
\end{cases}\end{aligned}\end{equation}
Then the volume form $dV$ is given by $dV = f(b)dV_{\alpha}$, where $dV_{\alpha}=\sqrt{det(a_{ij})}dx$ denotes the Riemannian volume form of $\alpha$, $T(s):=\phi(\phi-s\phi')^{n-2}[\phi-s\phi'+(b^2-s^2)\phi'']$.\end{proposition}\par

By Proposition \ref{xxc1} and Proposition \ref{xxg1}, we have
\begin{equation}\label{xg3}\begin{aligned}
\frac{\partial G^i}{\partial y^i}&=\frac{\partial
\bar{G}^i}{\partial
y^i}+\frac{2s_0}{(2s-1)^2}+\frac{6(b^2-s^2)}{(2s-1)(3s-2b^2-1)^2}s_0
-\frac{2s}{(2s-1)(3s-2b^2-1)}s_0\\
& +\frac{4(b^2-s^2)}{(2s-1)^2(3s-2b^2-1)}s_0
+(n+1)\frac{4s-1}{(2s-1)(3s-2b^2-1)}s_0\\
&+\frac{3(b^2-s^2)}{\alpha(3s-2b^2-1)^2}r_{00}
+(n+1)\frac{4s-1}{2\alpha(3s-2b^2-1)}r_{00}
-\frac{2}{(3s-2b^2-1)^2}r_{0};
\end{aligned}\end{equation}
and
\begin{equation}\label{xg4}\begin{aligned}
y^i\frac{\partial \ln \sigma_ F}{\partial x^i} &=y^i\frac{\partial
\ln \sigma_ {\alpha}}{\partial x^i}+\Lambda(r_0+s_0),\\
 \end{aligned}\end{equation}
where $\Lambda:=\frac{f'(b)}{bf(b)}$.\par
 Plugging \eqref{xg3} and \eqref{xg4} into \eqref{xg1}, we obtain
\begin{equation}\label{xg5}\begin{aligned}
S&=\frac{2s_0}{(2s-1)^2}+\frac{6(b^2-s^2)}{(2s-1)(3s-2b^2-1)^2}s_0
-\frac{2s}{(2s-1)(3s-2b^2-1)}s_0\\
& +\frac{4(b^2-s^2)}{(2s-1)^2(3s-2b^2-1)}s_0
+(n+1)\frac{4s-1}{(2s-1)(3s-2b^2-1)}s_0\\
&+\frac{3(b^2-s^2)}{\alpha(3s-2b^2-1)^2}r_{00}
+(n+1)\frac{4s-1}{2\alpha(3s-2b^2-1)}r_{00}
-\frac{2}{(3s-2b^2-1)^2}r_{0}\\
&+\Lambda(r_0+s_0).\\
 \end{aligned}\end{equation}

 \vspace{3mm}
{\it Proof of Theorem \ref{xxa4}}\par
 Assume that $S=0$. Multiplying both sides of \eqref{xg5} by $2\alpha^5(2s-1)^2(3s-2b^2-1)^2$, we obtain
\begin{equation}\label{xg6}\begin{aligned}
0&=24(2n+1)\beta^4r_{00}+\{
-4(13+8b^2+19n+8nb^2)\beta^3r_{00}+72\Lambda\beta^4r_0+72\Lambda\beta^4s_0\}\alpha\\
&+\{
2(19+32b^2+22n+20nb^2)\beta^2r_{00}-2(60+48b^2)\Lambda\beta^3r_0\\
&\quad  -24(1-2n+5\Lambda+4\Lambda b^2)\beta^3s_0\}\alpha^2\\
&+\{
-(11+40b^2+11n+16nb^2)\beta r_{00}+2(-8+37\Lambda+64\Lambda b^2+16\Lambda b^4)\beta^2r_0\\
&\quad  +2(12-26n-16nb^2+37\Lambda+64\Lambda b^2+16\Lambda b^4)\beta^2s_0\}\alpha^3\\
& +\{
(1+8b^2+n+2nb^2)r_{00}+4(4-5\Lambda-14\Lambda b^2-8\Lambda b^4)\beta r_0\\
&\quad  -2(5-8b^2-9n-12nb^2+10\Lambda+28\Lambda b^2+16\Lambda b^4)\beta s_0\}\alpha^4\\
&+\{
-(4-2\Lambda-8\Lambda b^2-8\Lambda b^4)r_0\\
&\quad +2(1-4b^2-n-2nb^2+\Lambda+4\Lambda b^2+4\Lambda b^4)s_0\}\alpha^5.\\
\end{aligned}\end{equation}\par
\eqref{xg6} is equivalent to the following
\begin{equation}\label{xg7}\begin{cases}\begin{aligned}
0&=24(2n+1)\beta^4r_{00}\\
&+\{ 2(19+32b^2+22n+20nb^2)\beta^2r_{00}-2(60+48b^2)\Lambda
\beta^3r_0-24(1-2n+5\Lambda +4\Lambda b^2)\beta^3s_0\}\alpha^2\\
& +\{
(1+8b^2+n+2nb^2)r_{00}+4(4-5\Lambda -14\Lambda b^2-8\Lambda b^4)\beta r_0\\
&\quad -2(5-8b^2-9n-12nb^2+10\Lambda +28\Lambda b^2+16\Lambda b^4)\beta s_0\}\alpha^4,\\
0&=\{
-4(13+8b^2+19n+8nb^2)\beta^3r_{00}+72\Lambda \beta^4r_0+72\Lambda \beta^4s_0\}\\
&+\{
-(11+40b^2+11n+16nb^2)\beta r_{00}+2(-8+37\Lambda +64\Lambda b^2+16\Lambda b^4)\beta^2r_0\\
&\quad +2(12-26n-16nb^2+37\Lambda +64\Lambda b^2+16\Lambda b^4)\beta^2s_0\}\alpha^2\\
&+\{ -(4-2\Lambda -8\Lambda b^2-8\Lambda b^4)r_0+2(1-4b^2-n-2nb^2+\Lambda +4\Lambda b^2+4\Lambda b^4)s_0\}\alpha^4.\\
\end{aligned}\end{cases}\end{equation}\par

From the first equation of \eqref{xg7}, we have
\begin{equation}\label{xg8}
r_{00}= c\alpha^2,
\end{equation}
for some function $c=c(x)$ on $M$. So $r_0 =  c\beta$.

Plugging \eqref{xg8} and $r_0 =  c\beta$ into \eqref{xg7}, we obtain
\begin{equation}\label{xg9}\begin{cases}\begin{aligned}
0=&
24c(1+2n-5\Lambda -4\Lambda b^2)\beta^4-24(1-2n+5\Lambda +4\Lambda b^2)\beta^3s_0\\
+&\{
2c(27+32b^2+22n+20nb^2-10\Lambda -28\Lambda b^2-16\Lambda b^4)\beta^2\\
&-2(5-8b^2-9n-12nb^2+10\Lambda +28\Lambda b^2+16\Lambda b^4)\beta
s_0\}\alpha^2
+c(1+8b^2+n+2nb^2)\alpha^4,\\
0=& 72\Lambda \beta^4( c\beta+s_0).\\
\end{aligned}\end{cases}\end{equation}
From the second equation of \eqref{xg9}, we have $c\beta+s_0=0$ for $n\geq2$. Differentiating both sides  of it with respect to $y^i$  yields $cb_i+s_i=0$. Contracting it with $b^i$ gives  $cb^2=0$. So $c=0$ and $s_0=0$. Thus $r_{00}=0,s_0=0$, i.e., $\beta$ is a constant Killing form.

Conversely, if $\beta$ is a constant Killing form, then $S=0$ by \eqref{xg5}. Thus we have completed the proof of Theorem \ref{xxa4}.\qed
\vspace{3mm}

By Theorem \ref{xxa4} and Theorem \ref{xxa2}, we can directly get Corollary \ref{xxa5}.

\vspace{3mm}

\vspace{1cm}
\section{Appendix: coefficients in \eqref{xd2}}\label{xxx6}
\begin{equation*}
\begin{aligned}
t_0&= 144(8n-11)\beta^{10}r_{00}^2;\\
\end{aligned}
\end{equation*}

\begin{equation*}
\begin{aligned}
t_1&= -96\{61n-82+(20n-26)b^2\}\beta^9r_{00}^2-432(2n-3)\beta^{10}r_
{00|0};\\
\end{aligned}
\end{equation*}

\begin{equation*}
\begin{aligned}
t_2&= 12\{1085n-1439+(792n-1032)b^2+64(n-1)b^4\}\beta^8r_{00}^2\\
&+1296\beta^{10}\overline{Ric}-288(8n-14)\beta^9r_0r_{00}+864\beta^9s_0r_{00}
+72\{63n-91+(24n-32)b^2\}\beta^9r_{00|0};\\
\end{aligned}
\end{equation*}

\begin{equation*}
\begin{aligned}
t_3&= -864(2n-1)\beta^9r_{0k}s^k_{\,\,\,0} -24\{697n-926+(852n-11
44)b^2+(152n-144)b^4\}\beta^7r_{00}^2\\
&-3456(2+b^2)\beta^9\overline{Ric} +96\{118n-205+(32n-44)
b^2\}\beta^8r_{00}r_0
-864\beta^9r_{00}r^k_{\,\,\,k}\\
&-48\{-16n+97+16(n-1)b^2\}
\beta^8r_{00}s_0-864\beta^9b^kr_{00|k}\\
&-24 \{435n-602+(354n-440)b^2+(48n-56)b^4\}\beta^8r_{00|0}+864\beta^9r_{0|0}-432 (2n-5)\beta^9s_{0|0};\\
\end{aligned}
\end{equation*}

\begin{equation*}
\begin{aligned}
t_4&=
144\{57n-22+(24n-8)b^2\}\beta^8r_{0k}s^k_{\,\,\,0}\\
&+3\{4606n-6255+(8400n-12080)b^2+(2480n-2272) b^4\}\beta^6r_{00}^2\\
& +216(15+4b^2)
(5+4b^2)\beta^8\overline{Ric}-32\{752n-1301+(440n-566)b^2+32(n-1)b^4\}\beta^7r_{00}
r_0\\
&+864(5+2b^2)\beta^8(r_{00}r^k_{\,\,\,k}+b^kr_{00|k})-576\beta^8rr_{00}\\
&+8
\{-413n+1322+(376n-664)b^2+64(n-1)b^4\}\beta^7r_{00}s_0\\
&+4\{3473n-4583+(4512n-5136)b^2+(1320n-1368)b^4+64(n-1)b^6\}\beta^7r_{00|0}\\
& -864(5+2b^2)\beta^8r_{0|0}+576
\beta^8r_0^2-1152(2n-3)\beta^8r_0s_0
+72\{57n-142+(24n-56)b^2\}\beta^8s_{0|0}\\
&-144(8n-21)\beta^8s_0^2-1296\beta^9s^k_{\,\,\,0|k};\\
\end{aligned}
\end{equation*}

\begin{equation*}
\begin{aligned}
t_5&=
-24\{699n-178+(636n-112)b^2+(96n-16)b^4\}\beta^7s^k_{\,\,\,0}r_{0k}\\
&-12\{643n-911+(1642n-2645)b^2+(712n-656)b^4\}\beta^5r_{00}^2\\
&-24
(917+1560b^2+672b^4+64b^6)\beta^7\overline{Ric}\\
&+16\{1814n-3143+(1712n-2024) b^2+(272n-224)b^4\}\beta^6r_{00}r_0\\
&-144 (65+56b^2+8b^4)\beta^7(b^kr_{00|k} +r_
{00}r^k_{\,\,\,k})+384(7+2b^2)\beta^7rr_{00}\\
& -4\{-1487n+3338+(1240n-
3952)b^2+(544n-736)b^4\}\beta^6r_{00}s_0\\
&-\{11854n-14857+
(21768n-22176)b^2+(10272n-9024)b^4+(1088n-896)b^6\}\beta^6r_{00|0}\\
&+144 (65+56b^2+8b^4)\beta^7r_{0|0}-384(7+2b^2)\beta^7r_0^2 +48\{2
12n-321+(64n-72)b^2\}\beta^7r_0s_0\\
& -12\{699n-1738+(636n-1456)b^2+(96n-208) b^4\}\beta^7s_{0|0}-324
\beta^7s_{0k}s^k_{\,\,\,0}\\
&+864
\beta^8(r_ks^k_{\,\,\,0}+r_{0k}s^k-b^ks_{0|k}-r^k_{\,\,\,k}s_0)+216(29+16b^2)\beta^8s^k_{\,\,\,0|k}-432(2n-3)\beta^8s^k_{\,\,\,0}s_k\\
&+48\{114n-239+(24n-52)b^2\}\beta^7s_0^2;\\
\end{aligned}
\end{equation*}

\begin{equation*}
\begin{aligned}
t_6&= 4\{4849n-516+(7116n+96)b^2+(2352n+96)b^4+128nb^6\}\beta^6s^k_
{\,\,\,0}r_{0k}\\
&+3/4\{3965n-
5929+(13592n-25576)b^2+(8096n-7936)b^4\}\beta^4r_{00}^2\\
&+(19225+46208b^2+32064b^4+6656b^6+256b^8) \beta^6\overline{Ric}\\&
-12\{1828n-3201+(2464n-2616)b^2+(640n-384)b^4\}\beta^5r_
{00}r_0\\
&+32(361+501b^2+156b^4+8b^6)\beta^6(r^k_{\,\,\,k}r_{00}+b^kr_{00|k})-32
(167+104b^2+8b^4)\beta^6r_{00}r\\
&+6\{-973n+1745+(776n-4336)
b^2+(656n-1216)b^4\}\beta^5r_{00}s_0\\
&+3/2\{4525n-5373+(10
968n-9696)b^2+(7392n-5088)b^4+(1280n-768)b^6\}\beta^5r_{00|0}\\
&-32(361+501b^2+156b^4+8b^6)
\beta^6r_{0|0}+32(167+104b^2+8b^4)\beta^6r_0^2\\
&-8\{2372n-3657+
(1568n-1656)b^2+(128n-96)b^4\}\beta^6s_0r_0+432(9+4b^2)\beta^7s_0r^k_
{\,\,\,k}\\
&+2\{4849n-12068+(711
6n-15936)b^2+(2352n-4896)b^4+(128n-256)b^6\}\beta^6s_{0|0}\\
&+432
(9+4b^2)\beta^7(b^ks_{0|k}-r_{0k}s^k-s^k_{\,\,\,0}r_k)-108(121+14
4b^2+32b^4)\beta^7s^k_{\,\,\,0|k}-576\beta^7s_0r\\
&+54(23+16b^2)\beta^6s^k_
{\,\,\,0}s_{0k}+72\{51n-73+(24n-32)b^2\}\beta^7s^k_{\,\,\,0}s_k\\
&-4\{2737n
-4424+(1240n-1208)b^2+(64n-128)b^4\}\beta^6s_0^2-324\beta^8s^j_{\,\,\,k}s^k_{\,\,\,j}-1296\beta^8\sigma;\\
\end{aligned}
\end{equation*}

\begin{equation*}
\begin{aligned}
t_7 &= -2(401+3504b^2+2400b^4+256b^6+7005n+14
652nb^2+7920nb^4+960nb^6)\beta^5s^k_{\,\,\,0}r_{0k}\\
&-3(-411-2650b^2-10
16b^4+263n+1170nb^2+916nb^4)\beta^3r_{00}^2\\
&-2(5651+17932b^2+17760b^4+6016b^6+512 b^8)\beta^5\overline{Ric}\\
&+6(-3225-
2988b^2-384b^4+1798n+3232nb^2+1216nb^4)\beta^4r_{00}r_0\\
&-2(4483+8880b^2+4512
b^4+512b^6)\beta^5(r_{00}r^k_{\,\,\,k}+b^kr_{00|k})+32(37+8b^2)(5+4b^2)\beta^5r_{00}r\\
&-3
(1817-8560b^2-3328b^4-1135n+968nb^2+1328nb^4)\beta^4r_{00}s_0\\
&-3/2(-1983-3960b^2-2232b^4-384b^6+1763n+5394nb^2+4848nb^4+12
16nb^6)\beta^4r_{00|0}\\
&+2 (4483+8880b^2+4512b^4+512b^6)\beta^5r_{0|0}-32(37+8b^2)
(5+4b^2)\beta^5r_0^2\\
 &-4(-
7769-5144b^2-560b^4+4884n+5280nb^2+960nb^4)\beta^5s_0r_0\\
&-(-
17531-32016b^2-15648b^4-1792b^6+7005n+14652nb^2+7920nb^4+960nb^6)\beta^5s_{0|0}\\
&+72(103+100b^2+16b^4) \beta^6(s^kr_{0k}+s^k_{\,\,\,0}r_k-s_0r^k_
{\,\,\,k}-b^ks_{0|k})\\
&+6
(2579+4944b^2+2400b^4+256b^6)\beta^6s^k_{\,\,\,0|k}-216(9+14b^2+4b^4)\beta^5s^k_{\,\,\,0}
s_{0k}\\
&-12(-739-
688b^2-112b^4+546n+564nb^2+96nb^4)\beta^6s^k_{\,\,\,0}s_k\\
&+4
(-3622+656b^2+320b^4+3003n+2184nb^2+240nb^4)\beta^5s_0^2\\
&+8(-7769-514
4b^2-560b^4+4884n+5280nb^2+960nb^4)\beta^5s_0r_0+96(25+8b^2)
\beta^6s_0r\\
&+
108(13+8b^2)\beta^7s^j_{\,\,\,k}s^k_{\,\,\,j}+864(5+4b^2)\beta^7\sigma+432\beta^7s^ks_k;\\
\end{aligned}
\end{equation*}

\begin{equation*}
\begin{aligned}
t_8&=2(769+4596b^2+4512b^4+896b^6+3285n+9126nb^2+7128nb^4+1440nb^6)\beta^4r_{0k}s^k_{\,\,\,0}\\
&+3/2(-145-1436b^2-
684b^4+93n+520nb^2+518nb^4)\beta^2r_{00}^2\\
&+(4535+18184b^2+24024b^4+11
776b^6+1664b^8)\beta^4\overline{Ric}\\
&-12
(-539-509b^2+16b^4+288n+660nb^2+336nb^4)\beta^3r_{00}r_0\\
&+2(2273+6006b^2+4416b^4+832b^6)\beta^4 (r_{00}r^k_{\,\,\,k}+b^kr_{00|k}) -4 (1001+1472b^2+416b^4)\beta^4r_{00}r\\
&+12(165-1293b^2-672b^4-100n+110nb^2+206nb^4)\beta^3r_{00}s_0\\
&+3(-247-492b^2-186b^4+16b^6+231n+864nb^2+990nb^4+336nb^6)\beta^3r_
{00|0}\\
&-2(2273+6006b^2+4416b^4+832b^6)\beta^4r_{0|0}+4(1001+1472b^2+416b^4)\beta^4r_0^2\\
&-2(-
10175-8360b^2-1040b^4+6084n+9504nb^2+2880nb^4)\beta^4r_0s_0\\
&+4(1961+3108b^2+1104b^4+64b^6)\beta^5(r^k_
{\,\,\,k}s_0+b^ks_{0|k}-r_{0k}s^k-s^k_{\,\,\,0}r_k)\\
&+(-8323-19428b^2-13152b^4-2432b^6+3285n+9126nb^2+7128nb^4+14
40nb^6)\beta^4s_{0|0}\\
&-16(718+1961b^2+1554b^4+368b^6+16b^8)\beta^5s^k_
{\,\,\,0|k}+6(269+696b^2+456b^4+64b^6)\beta^4s^k_ {\,\,\,0}s_{0k}\\
&+2(-4075-5856b^2-2064b^4-128b^6+3211n+5424nb^2+2064nb^4+12
8nb^6)\beta^5s^k_{\,\,\,0}s_k\\
&-(-
7009+13792b^2+7744b^4+512b^6+7881n+8088nb^2+1392nb^4)\beta^4s_0^2\\
&-27(95+128b^2+32b^4)\beta^6s^j_{\,\,\,k}s^k_{\,\,\,j}-
216(29+48b^2+16b^4)\beta^6\sigma-864(2+b^2)\beta^6 s^ks_k\\
&-16
(259+184b^2+16b^4)\beta^5rs_0;\\
\end{aligned}
\end{equation*}

\begin{equation*}
\begin{aligned}
t_9&= -4(2
11+1428b^2+1938b^4+608b^6+501n+1764nb^2+1854nb^4+552nb^6)\beta^3s^k_{\,\,\,0}
r_{0k}\\
&-3(-7-114b^2-
68b^4+5n+34nb^2+42nb^4)\beta r_{00}^2\\
&-4
(1+2b^2)(307+894b^2+756b^4+176b^6)\beta^3\overline{Ric}\\
&+12
(-116-99b^2+56b^4+58n+164nb^2+108nb^4)\beta^2r_{00}r_0\\
&-4(377+1272b^2+1266b^4+352b^6)\beta^3(r_{00}r^k_{\,\,\,k}+b^kr_{00|k}-r_{0|0})\\
&+16(106+211b^2+88b^4)\beta^3rr_{00}
-6(85-934b^2-636b^4-41n+72nb^2+158nb^4)\beta^2r_{00}s_0\\
&-3(-40-68b^2+34b^4+56b^6+39n+174nb^2+246nb^4+10
8nb^6)\beta^2r_{00|0}\\
&-16(106+211b^2+88b^4)\beta^3r_0^2+16(-530-479b^2-
17b^4+294n+618nb^2+276nb^4)\beta^3s_0r_0\\
&+8(481+568b^2+112b^4)\beta^4 rs_0\\
&-2(-1297-
3660b^2-3126b^4-800b^6+501n+1764nb^2+1854nb^4+552nb^6)\beta^3s_{0|0}\\
&+4(1261+2886b^2+1704b^4+224b^6)\beta^4(s^kr_{0k}+s^k_{\,\,\,0}r_k-r^k_{\,\,\,k}s_0-b^ks_{0|k})\\
&+2 (2779+10088b^2+11544b^4+4544b^6+448b^8)\beta^4s^k_{\,\,\,0|k}\\
&-4(1+2b^2)(193+342b^2+132b^4+8b^6) \beta^3s^k_{\,\,\,0}s_{0k}\\
&-2(-2245-4200b^2-
2280b^4-320b^6+1897n+4614nb^2+2928nb^4+416nb^6)\beta^4s^k_{\,\,\,0}s_k\\
&+4
(-549+3536b^2+2828b^4+368b^6+789n+1050nb^2+240nb^4)\beta^3s_0^2\\
 &+6
(431+948b^2+528b^4+64b^6)\beta^5s^j_{\,\,\,k}s^k_{\,\,\,j}+24(5+4b^2)
(43+76b^2+16b^4)\beta^5\sigma\\
&+36
(79+88b^2+16b^4)\beta^5s^ks_k;\\
\end{aligned}
\end{equation*}

\begin{equation*}
\begin{aligned}
t_{10}&=4(59+480b^2+870b^4+400b^6+96n+414
nb^2+558nb^4+228nb^6)\beta^2r_{0k}s^k_{\,\,\,0}\\
&+3/4\{n-1+(8n-32)b^2+(12n-24)b^4\}r_{00}^2+(215+404b^2+164b^4)(1+2b^2)
^2\beta^2\overline{Ric}~~~~~~~~~~~~~~~~~~~~~~~~~~~~~~~~~~~~~~~~~~~~~~~~~~~~~~~~~~~~~~~~~~~~~~~~~~~~~~~~~~~~~~~~~~~~~~~~~~~~~~~~~~~~~~~~~~~~~~~~~~~~~~~~~~~~~~~~~~~~~~~~~~~~~~~~\\~~~~~~~~~~~~~~~~~~~~~~~~~~~~~~~~~~~~~~~~~~~~~~~~~~~~~~~~~~~~~~~~~~~~~~~~~~~~~~~~~~~~~~~~~~~~~~~~~~~~~~~~~~~~~~~~~~~~~~~~~~~~~~~~~~~~~~~~~~~~~~~~~~~~~~~~
&- 4(-44-29b^2+64b^4+20n+68nb^2+56nb^4)\beta
r_{00}r_0\\
&-8(55+142b^2+82b^4)\beta^2r_{00}r+2(43
-554b^2-488b^4-13n+44nb^2+104nb^4)\beta r_{00}s_0\\
&+4(1+2b^2)(79+172b^2+82b^4)\beta^2(b^kr_{00|k}+r_
{00}r^k_{\,\,\,k})\\
&+1/2(1+2b^2)(-23+22b^2+64b^4+23n+74nb^2+56nb^4)\beta r_{00|0}\\
&-4(1+2b^2) (79+172b^2+82b^4)\beta^2r_{0|0}
+8(55+142b^2+82b^4)\beta^2r_0^2\\
&-4(-553-
496b^2+158b^4+276n+744nb^2+456nb^4)\beta^2s_0r_0\\
&+8(253+780b^2+678b^4+152b^6)\beta^3(s_0r^k_ {\,\,\,k}+b^ks_{0|k}
-r_{0k}s^k-s^k_ {\,\,\,0}r_k)\\
&+2(-
257-840b^2-834b^4-256b^6+96n+414nb^2+558nb^4+228nb^6)\beta^2s_{0|0}\\
&-4(1+2b^2)(439+ 1146b^2+828b^4+152b^6)\beta^3s^k_{\,\,\,0|k}-32(65+113b^2+38b^4)\beta^3s_0r\\
&+2(107+116b^2+20b^4)(1+2b^2)^2\beta^2s^k_{\,\,\,0}s_{0k}-4 (625+1158b^2+480b^4+32b^6)\beta^4s^ks_k\\
&+4(-383-804b^2-510b^4-112b^6+347n+1128nb^2+10
50nb^4+256nb^6)\beta^3s_ks^k_{\,\,\,0}\\
&-2(-
255+3368b^2+3668b^4+752b^6+375n+600nb^2+150nb^4)\beta^2s_0^2\\
&-(1579+5000b^2+4632b^4+12 80b^6+64b^8)\beta^4s^j_{\,\,\,k}s^k_{\,\,\,j}\\
& - (2641+9344b^2+10176b^4+3584b^6+256b^8)\beta^4\sigma
;\\
\end{aligned}
\end{equation*}

 \begin{equation*}
\begin{aligned}
t_{11}&= -2(1+2b^2)(17+134b^2+128b^4+21n+66nb^2+48nb^4)\beta
s^k_{\,\,\,0}
r_{0k}\\
&+2(1+2b^2)(-5+8b^2+2n+4nb^2) r_{00}r_0
-2(19+20b^2)(1+2b^2)^2\beta (r_{00}r^k_{\,\,\,k}+b^kr_{00|k})\\
& +16 (4+5b^2)(1+2b^2)\beta rr_{00}
-(7-92b^2-104b^4-n+8nb^2+20nb^4)r_{00}s_0\\
&-2(11+10b^2)(1+2b^2)^3\beta\overline{Ric}
-1/2(1+2b^2)^2(-1+4b^2+n+2nb^2)
r_{00|0}\\
&+2(19+20b^2)(1+2b^2)^2\beta r_{0|0} -16(4+5b^2)(1+2b^2)\beta
r_0^2\\
&+4(-83-68b^2+88b^4+36n+120nb^2+96nb^4)\beta r_0s_0
+16(41+98b^2+50b^4)\beta^2rs_0\\
&-(1+2b^2)(-59-98b^2-32b^4+21n+66nb^2+48nb^4)\beta s_{0|0}\\
&+16(1+2b^2)(31+61b^2+25b^4)\beta^2( r_{0k}s^k +s^k_{\,\,\,0}r_k-r^k_{\,\,\,k}s_0-b^ks_{0|k})\\
& +2(175+292b^2+100b^4)(2b^2+1)^2\beta^2s^k_ {\,\,\,0|k}
-16(2+b^2)(1+2b^2)^3\beta s^k_{\,\,\,0}s_{0k}\\
&-4(-80-156b^2-42b^4+8b^6+77n+318nb^2+402nb^4+148nb^6)\beta^2s_ks^k_{\,\,\,0}\\
&+4(-23+392b^2+560b^4+160b^6+
24n+42nb^2+6nb^4)\beta s_0^2\\
&+24(1+2b^2) (25+56b^2+32b^4+4b^6)\beta^3s^j_{\,\,\,k}s^k_{\,\,\,j}
+4(5+4b^2)(1+2b^2)
(43+76b^2+16b^4)\beta^3\sigma\\
&+24(53+144b^2+102b^4+16b^6)\beta^3s^ks_k;\\
\end{aligned}
\end{equation*}

\begin{equation*}
\begin{aligned}
t_{12}&= 2(1+2b^2)^2(1+8b^2+n+2nb^2)r_{0k}s^k_{\,\,\,0}
+(1+2b^2)^4\overline{Ric}
 +2(1+2b^2)
^3(r_{00}r^k_{\,\,\,k}+b^kr_{00|k})\\
& -4(1+2b^2)^2rr_{00} -2 (1+2b^2)^3r_{0|0} +4(1+2b^2)^2r_0^2
-2(1+2b^2)\{4n-11+(8n+14)b^2\} r_0s_0\\
& +4(1+2b^2)^2(17+16b^2)\beta (s_0r^k_{\,\,\,k} +b^ks_{0|k}
-s^kr_{0k} -s^k_{\,\,\,0}r_k)-8(5+4b^2)(1+2b^2)^3\beta s^k_{\,\,\,0|k}\\
& -16(1+2b^2)(7+8b^2) \beta s_0r
+(1+2b^2)^2(-3+n+2nb^2)s_{0|0}\\
& +2(1+2b^2)^4s^k_{\,\,\,0}s_{0k} +2
(1+2b^2)\{19n-19+(58n+14)b^2+(40n+32)b^4\}\beta s^k_{\,\,\,0}s_k\\
&+\{-5n+9-(8n+144)b^2+(4n-264)b^4-96b^6\}s_0^2-8(1+2b^2)(47+80b^2+26b^4)\beta^2s^ks_k\\
&-(139+196b^2+52b^4)(1+2b^2)^2\beta^2s^j_{\,\,\,k}s^k_{\,\,\,j}
-6(29+48b^2+16b^4)(1+2b^2)^2\beta^2\sigma;\\
\end{aligned}
\end{equation*}

\begin{equation*}
\begin{aligned}
t_{13}&=8(1+2b^2)^2rs_0+4(1+2b^2)
^3(s^kr_{0k}+s^k_{\,\,\,0}r_k-s_0r^k_{\,\,\,k}-b^ks_
{0|k})+2(1+2b^2)^4s^k_{\,\,\,0|k}\\~~~~~~~~~~~~~~~~~~~~~~~~~~~~~~~~~~~~~~~~~~~~~~~~~~~~~~~~~~~~~~~~~~~~~~~~~~~~~~~~~~~~~~~~~~~~~~~~~~~~~~~~~~~~~~~~~~~~~~~~~~~~~~~~~~~~~~~~~~~~~~~~~~~~~~~~~~~~~~~~~~~~~~~~~~~~~~~~~~~~~~~~~~~~~~~~~~~~~~~~~~~~~~~~~~~~~~~~~~~~~~~~~~~~~~~~~~~~~~~~~~~~~~~~~~~~~~~~~~~~~~~~~~~~~~~~~~~~~~~~~~~~~~~~~~~~~~~~~~~~~~~~~~~~~~~~~~~~~~~~~~~~~~~~~~~~~~~~~~~~~~~~~~~~~~~~~~~~~~~~~~~~~~~~~~~~~~~~~~~~~~~~~~~~~~~~~~~~~~~~~~~~~~~~~~~~~~~~~~~~~~~~~~~~~~~~~~~~~~~~~~~~~~~~~~~~~~~~~~~~~~~~~~~~~~~~~~~~~~~~~~~~~~~~~~~~~~~~~~~~~~~~~~~~~~~~~~~~~~~~~~~~~~~~~~~~~~~~~~~~~~~~~~~~~~~~~~~~~~~~~~~~~~~~~~~~~~~~~~~~~~~~~~~~~~~~~~~~~~~~~~~~~~~~~~~~~~~~~~~~~~~~~~~~~~~~~~~~~~~~~~~~~~~~~~~~~~~~~~~~~~~~~~~~~~~~~~~~~~~~~~~~~~~~~~~~~~~~~~~~~~~~~~~~~~~~~~~~~~~~~~~~~~~~~~~~~~~~~~~~~~~~~~~~~~~~~~~~~~~~~~~~~~~~~~~~~~~~~~~~~~~~~~~~~~~~~~~~~~~~~~~~~~~~~~~~~~~~~~~~~~~~~~~~~~~~~~~~~~~~~~~~~~~~~~~~~~~~~~~~~~~~~~~~~~~~~~~
&-2(1+2b^2)^2(-1+4b^2+n+2nb^2)s^k_{\,\,\,0}s_k+6(1+2b^2)^3(3+2b^2)\beta
s^j_{\,\,\,k}s^k_{\,\,\,j}+4(5+4b^2)(1+2b^2)^3\beta\sigma\\
&+12(1+2b^2)^2(5+4b^2)\beta
s^ks_k;\\
t_{14}&=
-(1+2b^2)^4s^j_{\,\,\,k}s^k_{\,\,\,j}-(1+2b^2)^4\sigma-4(1+2b^2)^3s^ks_k.
\end{aligned}
\end{equation*}

\vspace{6mm}

  Yi-Bing Shen

  Center of Math. Science,

  Yuqun Campus, Zhejiang University,

  Hangzhou 310027, China,

  {\it Email: yibingshen$@$zju.edu.cn}

\vskip 5mm

Xiaoling Zhang

 Department of Mathematics,

  Zhejiang University,

  Hangzhou 310027, China,

  College of Mathematics and Systems Science,

 Xinjiang University,

 Urumqi 830046, China,

 {\it Email: xlzhang@ymail.com}

 \vskip 5mm


\vspace{1cm}
\begin{thebibliography}{23}
\bibitem{aik} T.Aikou, M.Hashiguchi and K.Yamaguchi, \textsl{On Matsumoto's Finsler space with time measure,} Rep. Fac. Sci. Kagoshima Univ. (Math. Phys. Chem), 23(1990), 1-12.

\bibitem{AZ} H.Akbar-Zadeh, \textsl{Generalized Einstein manifolds}, J. Geom. and Phys., {\bf 17}(1995), 342-380.

\bibitem{bao1} D.Bao and C.Robles, \textsl{Ricci and flag curvatures in Finsler geometry}, in ''A Sampler of Finsler  Geometry'', MSRI series {\bf 50}, Camb. Univ. Press, 2004, 197-259.

\bibitem{cheng} X.Cheng, Z.Shen and Y.Tian, \textsl{A Class of Einstein $(\alpha,\beta)$-metrics,} Israel Journal of Mathematics, accepted.

\bibitem{CHEN} S.S.Chern and Z.Shen., \textsl{Riemann-Finsler geometry}, World Scientific, 2005.

\bibitem{li} B.Li, \textsl{Projectively flat Matsumoto metric and its approximation,} Acta Mathematica Scientia 2007, 27B(4), 781-789.


\bibitem{mat} M.Matsumoto, \textsl{A slope of a mountain is a Finsler surface with respect ot time measure,} J. Math. Kyoto Univ., 29(1989), 17-25.

\bibitem{park} H.S.Park, I.Y.Lee and C.K.Park, \textsl{Finsler space with the general approximate Matsumoto metric,} Indian J. pure and appl. Math., 34(1)(2002), 59-77.

\bibitem{raf2} M.Rafie-Rad and B.,Rezaei, \textsl{Matsumoto metrics of constant flag curvature are trivial,} Results. Math., Online First, 2011, Springer Basel AG, DOI 10.1007/s00025-011-0210-1.

\bibitem{raf1} M.Rafie-Rad and B.,Rezaei, \textsl{On Einstein Matsumoto metrics,} Nonlinear Anal. Real World Appl., Vol. 13, Issue 2, 2012, 882-886.


\bibitem{rez} B.Rezaei, A.Razavi and N.Sadeghzadeh, \textsl{ON EINSTEIN $(\alpha,\beta)$ -METRICS*,} Iranian Journal of Science Technology, Transaction A. Vol. 31. No. A4 Printed in The Islamic Republic of Iran, 2007.


\bibitem{zhang} X.Zhang. and Yi.Shen., \textsl{On Einstein Kropina metrics,} to appear in Differential geometry and its Applications.

\bibitem{zhou} L.Zhou, \textsl{A local classification of a class of $(\alpha,\beta)$-metrics with constant flag curvature,} Differential geometry and its Applications, 28(2010), 170-193.

\bibitem{bacso} S.B\'{a}cs\'{o}, X.Cheng and Z.Shen., \textsl{Curvature properties of $(\alpha,\beta)$-metrics,} Advanced Studies in Pure Mathematics, Math. Soc. of Japan, 48(2007), 73-110.
\end{thebibliography}
 \end{document}